\documentclass{birkjour}
 \newtheorem{thm}{Theorem}[section]
 
 \newtheorem{lem}[thm]{Lemma}
 
 \theoremstyle{definition}
 
 \theoremstyle{remark}
 \newtheorem{rem}[thm]{Remark}
 
 \numberwithin{equation}{section}

\numberwithin{equation}{section}

\newenvironment{altproof}[1]
{\noindent%\addvspace{0.3cm}
{\em Proof of {#1}}.}
{\nopagebreak\mbox{}\hfill $\Box$\par\addvspace{0.5cm}}

\newcommand\mytop[2]{\genfrac{}{}{0pt}{}{#1}{#2}}
\newcommand\set[1]{\left\{\,#1\,\right\}}  % Menge
\newcommand\abs[1]{\left|#1\right|} % Absolutbetrag
\newcommand\norm[1]{\left\Vert#1\right\Vert} % Skalarprodukt
\newcommand\Sym{\text {Sym}}

\newcommand{\veps}{\varepsilon}

\newcommand{\om}{\omega}
\newcommand{\si}{\sigma}
\newcommand{\Ga}{\Gamma}
\newcommand{\De}{\Delta}
\newcommand{\La}{\Lambda}
\newcommand{\Om}{\Omega}
\newcommand{\Si}{\Sigma}

\newcommand{\R}{\mathbb{R}}
\newcommand{\Z}{\mathbb{Z}}

\newcommand{\cC}{\mathcal{C}}

\newcommand{\cF}{\mathcal{F}}

\newcommand{\cM}{\mathcal{M}}

\newcommand{\cO}{\mathcal{O}}

\newcommand{\cU}{\mathcal{U}}

\newcommand{\fJ}{{\mathfrak J}}

\newcommand{\conf}{\cF_N(\Om)}

\newcommand{\pa}{\partial}
\def\id{\mathrm{id}}

\newcommand{\beq[1]}{\begin{equation}\label{eq:#1}}
\newcommand{\eeq}{\end{equation}}

\newcommand{\beql}{\begin{equation}\label}

\newcommand{\beqq}{\begin{equation*}}
\newcommand{\eeqq}{\end{equation*}}

\newcommand{\bal}{\begin{aligned}}
\newcommand{\eal}{\end{aligned}}

\newcommand{\bdpm}{\begin{displaymath}}
\newcommand{\edpm}{\end{displaymath}}

\begin{document}

%-------------------------------------------------------------------------
% editorial commands: to be inserted by the editorial office
%
%\firstpage{1} \volume{228} \Copyrightyear{2004} \DOI{003-0001}
%
%
%\seriesextra{Just an add-on}
%\seriesextraline{This is the Concrete Title of this Book\br H.E. R and S.T.C. W, Eds.}
%
% for journals:
%
%\firstpage{1}
%\issuenumber{1}
%\Volumeandyear{1 (2004)}
%\Copyrightyear{2004}
%\DOI{003-xxxx-y}
%\Signet
%\commby{inhouse}
%\submitted{March 14, 2003}
%\received{March 16, 2000}
%\revised{June 1, 2000}
%\accepted{July 22, 2000}
%
%
%
%---------------------------------------------------------------------------
%Insert here the title, affiliations and abstract:
%

\title[Periodic solutions of Hamiltonian systems of $N$-vortex type]
 {Periodic solutions of singular first-order \\Hamiltonian systems of N-vortex type}

%----------Author 1
\author[Thomas Bartsch]{Thomas Bartsch}

\address{%
Mathematisches Institut\\
Universit\"at Giessen\\
Arndtstr.\ 2\\
35392 Giessen\\
Germany}

\email{Thomas.Bartsch@math.uni-giessen.de}

%\thanks{This work was completed with the support of our \TeX-pert.}
%----------Author 2
%\author{A Second Author}
%\address{The address of\br
%the second author\br
%sitting somewhere\br
%in the world}
%\email{dont@know.who.knows}
%----------classification, keywords, date
\subjclass{Primary 37J45; Secondary 37N10, 76B47}

\keywords{$N$-vortex dynamics; singular first order Hamiltonian systems; periodic solutions}

%\date{April 1, 2016}
%----------additions
%\dedicatory{Dedicated to Ernst-Ulrich Gekeler}
%%% ----------------------------------------------------------------------

\begin{abstract}
We are concerned with the dynamics of $N$ point vortices $z_1,\dots,z_N\in\Om\subset\R^2$ in a planar domain. This is described by a Hamiltonian system
\[
\Ga_k\dot{z}_k(t)=J\nabla_{z_k} H\big(z(t)\big),\quad k=1,\dots,N,
\]
where $\Ga_1,\dots,\Ga_N\in\R\setminus\{0\}$ are the vorticities, $J\in\R^{2\times2}$ is the standard symplectic $2\times2$ matrix, and the Hamiltonian $H$ is of $N$-vortex type:
\[
H(z_1,\dots,z_N)
 = -\frac1{2\pi}\sum_{\mytop{j,k=1}{j\ne k}}^N \Ga_j\Ga_k\log\abs{z_j-z_k} - \sum_{j,k=1}^N\Ga_j\Ga_kg(z_j,z_k).
\]
Here $g:\Om\times\Om\to\R$ is an arbitrary symmetric function of class $\cC^2$, e.g.\ the regular part of a hydrodynamic Green function.
%Thus $H$ becomes singular as $|z_j-z_k|\to0$, for some $j\ne k$, or $z_k\to\pa\Om$ for some $k$.
Given a nondegenerate critical point $a_0\in\Om$ of $h(z)=g(z,z)$ and a nondegenerate relative equilibrium $Z(t)\in\R^{2N}$ of the Hamiltonian system in the plane with $g=0$, we prove the existence of a smooth path of periodic solutions $z^{(r)}(t)=\big(z^{(r)}_1(t),\dots,z^{(r)}_N(t)\big)\in\Om^N$, $0<r<r_0$, with $z^{(r)}_k(t)\to a_0$ as $r\to0$. In the limit $r\to0$, and after a suitable rescaling, the solutions look like $Z(t)$.
\end{abstract}

%%% ----------------------------------------------------------------------
\maketitle
%%% ----------------------------------------------------------------------
%\tableofcontents
\section{Introduction}
Let $\Om\subset\R^2$ be a smooth domain, and let $g:\Om\times\Om\to\R$ be symmetric and of class $\cC^2$. We are interested in the existence of periodic solutions $z(t)=(z_1(t),\dots,z_N(t))\in\Om^N$ of first order Hamiltonian systems of the form
\[
\Ga_k\dot{z}_k(t) = J\nabla_{z_k} H\big(z(t)\big),\quad k=1,\dots,N. \tag{$HS$}
\]
Here $J=\begin{pmatrix}0&1\\-1 & 0\end{pmatrix}$ is the standard symplectic matrix in $\R^2$. The Hamilton function $H$ is of the form
$H(z) = H_0(z) - F(z)$ with
\[
H_0(z) = -\frac1{2\pi}\sum_{{j,k=1}{j\ne k}}^N \Ga_j\Ga_k\log|z_j-z_k|
\]
the Kirchhoff-Onsager functional, and
\[
F(z) = \sum_{j,k=1}^N \Ga_j\Ga_k g(z_j,z_k).
%     = \sum_{\mytop{j,k=1}{j\ne k}}^N \Ga_j\Ga_k g(z_j,z_k) + \sum_{k=1}^N\Ga_k^2h(z_k)
\]
Here $\Ga_1,\dots,\Ga_N\in\R\setminus\{0\}$ are the vorticities. Thus setting
\[
G(w,z)=-\frac1{2\pi}\log|w-z|-g(w,z)
\]
and $h(z):=g(z,z)$ the Hamiltonian can also be written as
\[
H(z) = \sum_{\mytop{j,k=1}{j\ne k}}^N \Ga_j\Ga_k G(z_j,z_k) - \sum_{k=1}^N\Ga_k^2h(z_k).
\]
It is defined on the configuration space
\[
\conf=\{z\in\Om^N: z_j\ne z_k \text{ for }j\ne k\}.
\]
of $N$ different points in $\Om$. Obviously $H (z_1,\dots,z_N)$ becomes singular if $|z_j-z_k|\to0$ for some $j\ne k$ or if $|h(z_k)|\to\infty$ for some $k$. In the classical $N$-vortex problem $G$ is a hydrodynamic Green's function (see \cite{Flucher-Gustafsson:1997}) with regular part $g$ and $h$ is the Robin function. Then we have $|h(z_k)|\to\infty$ as $z_k\to\pa\Om$.

Hamiltonian systems like ($HS$) appear as singular limit equations in a variety of problems from mathematical physics. In fluid dynamics ($HS$) is derived from the Euler equation in vorticity form when one makes a point vortex ansatz. In this setting $H$ is the Kirchhoff-Routh path function. Other applications of ($HS$) are models for superconductivity where $H$ appears as renormalized energy for Ginzburg-Landau vortices, or the dynamics of a magnetic vortex system in a thin ferromagnetic film modeled by the Landau-Lifshitz-Gilbert equation. We refer the reader to \cite{Majda-Bertozzi:2001,Marchioro-Pulvirenti:1994,Saffmann:1995} for fluid dynamics, \cite{Bethuel-etal:1994, Colliander-Jerrard:1998,Lin-Xin:1998}) for the GLS equation, and to \cite{Kurzke-etal:2011} for the LLG equation. An introduction to point vortex dynamics is \cite{Newton:2001}.

There is a lot of work about ($HS$) for special domains where the Green's function is explicitly known. Most papers actually deal with the case when $\Om=\R^2$ is the plane so that boundary terms do not appear, i.e.\ $H=H_0$. In particular, periodic solutions of
\[
\Ga_k\dot{z}_k(t) = J\nabla_{z_k} H_0\big(z(t)\big),\quad k=1,\dots,N, \tag{$HS_0$}
\]
in the plane have been investigated by many mathematicians, see e.g.\ \cite{Aref-etal:2003} and the references therein. For a general domain even the existence of equilibria is difficult to prove; see \cite{Bartsch-Pistoia:2015,Bartsch-Pistoia-Weth:2010,Kuhl:2015,Kuhl:2016} for recent results in this direction. Concerning periodic solutions of ($HS$) in a general domain the only papers we are aware of are \cite{Bartsch-Dai:2016,Bartsch-Gebhard:2016}. In fact, although there is a lot of work on singular second order Hamiltonian systems motivated by the $N$-body problem, there are very few papers on singular first order Hamiltonian systems. Earlier papers, e.g.\ \cite{Carminati-Sere-Tanaka:2006,Tanaka:1996}, deal with first order Hamiltonians that are modeled after the $N$-body Hamiltonian and do not apply to vortex type Hamiltonian systems. In the present paper we shall give a new direct proof of a recent result from \cite{Bartsch-Gebhard:2016} that has been derived in \cite{Bartsch-Gebhard:2016} in an indirect way using a special equivariant degree theory. Our new proof is constructive and based on the contracting mapping principle, hence it may be implemented to compute the solutions numerically. This direct approach could also be useful for further investigations of the dynamics near the family of periodic solutions that we obtain.

\section{Main theorem}\label{sec:theorem}

In order to state our result we need to introduce some notation. A periodic relative equilibrium $Z(t)\in\cF_N(\R^2)$ with center of vorticity at $0$ is a solution of ($HS_0$) of the form
\beq[rel-equilib]
Z(t)=e^{-\om Jt}z,\quad \om\in\R\setminus\{0\},\ z\in\cF_N(\R^2),
\eeq
where $e^{-\om Jt}\in\R^{2\times2}$ acts on $z=(z_1,\dots,z_N)\in\cF_N(\R^2)\subset(\R^2)^N$ by multiplication on each component. Such a relative equilibrium $Z$ is non-degenerate provided the linearized system
\beq[R2lin]%\tag{$LH_{0}$}
M_\Gamma\dot{w} = J_N\big(H_0''(Z(t))\big)w
\eeq
has exactly three linearly independent $\frac{2\pi}{\abs{\om}}$-periodic solutions. This is the minimal possible dimension because $H_0$ is invariant under translations and rotations. Clearly $Z$ as in \eqref{eq:rel-equilib} is a non-degenerate $\frac{2\pi}{\abs{\om}}$-periodic equilibrium iff $Z_\om(t):=\sqrt{|\om|}Z(t/|\om|)$ is a non-degenerate $2\pi$-periodic equilibrium. We can therefore assume that $Z$ is $2\pi$-periodic, i.e.\ $\abs{\om}=1$. We shall work on the space
\[
H^1_{2\pi}:=\set{z:\R\to\R^{2N}\mid z(t+2\pi)=z(t) \forall t,\ z\in H^1_{loc}}
\]
of absolutely continuous $2\pi$-periodic functions $z:\R\to\R^{2N}$. The group $S^1=\R/2\pi\Z$ acts on $H^1_{2\pi}$ as follows:
\[
\theta*z(t) := z(t+\theta)\quad\text{ for }\theta\in S^1,\ z\in H^1_{2\pi}.
\]

\begin{thm}\label{thm:existence}
Let $Z$ be a nondegenerate periodic relative equilibrium of ($HS_0$) as in \eqref{eq:rel-equilib}, and let $a_0\in\Om$ be a nondegenerate critical point of $h:\Om\to\R$. If $\Ga:=\Ga_1+\ldots+\Ga_N \ne 0$ then the following holds.
\begin{enumerate}
\item[\rm a)] There exists $r_0>0$ and a $C^1$-map
    \[
    (0,r_0)\to H^1_{2\pi},\quad r\mapsto u^{(r)},
    \]
    with $u^{(r)}\to Z$ as $r\to0$, and such that $z^{(r)}(t):=a_0+ ru ^{(r)}(t/r^2)$ is a periodic solution of $(HS)$ for $0<r<r_0$.
\item[\rm b)] There exists an $S^1$-invariant neighborhood $\cU \subset H^1_{2\pi}$ of $S^1*Z$ such that if $u\in\cU$ generates a solution $z(t)=a_0+ru(t/r^2)$ of $(HS)$ for some $0<r<r_0$, then $u=\theta*u^{(r)}$ for some $\theta\in S^1$.
\end{enumerate}
\end{thm}

\begin{rem}\label{rem:main1}\rm
  a) Suppose $g$ is the regular part of the Dirichlet Green function in $\Om$ so that $h(z)\to\infty$ as $z\to\pa\Om$. In a bounded convex domain the Robin function $h$ is strictly convex and the (unique) minimum is nondegenerate; see \cite[Theorem 3.1]{Caffarelli-Friedman:1985}. There are bounded domains with an arbitrarily large number of critical points of $h$; these may even be simply connected, e.g.\ dumbbell shaped. Moreover, for a generic bounded domain all critical points of $h$ are non-degenerate; see \cite{Micheletti-Pistoia:2014}.

  b) A simple example of a nondegenerate relative equilibrium in the plane consists of two vortices $z_1,z_2$ with vorticities $\Ga_1,\Ga_2\ne0$ and such that $\Ga_1+\Ga_2\ne 0$. These rotate rigidly around their center of vorticity $\frac{\Ga_1z_1+\Ga_2z_2}{\Ga_1+\Ga_2}$. Explicitly one may choose $z_1,z_2\in\R^2\setminus\{0\}$ such that $\Ga_1z_1+\Ga_2z_2=0$. Then $Z(t):=\big(e^{-\om Jt}z_1,e^{-\om Jt}z_2\big)$ with $\om=\frac{\Ga_1+\Ga_2}{\pi|z_1-z_2|^2}$ solves ($HS_0$). Such an equilibrium is always non degenerate.

  Three vortices with vorticities $\Ga_1,\Ga_2,\Ga_3\ne 0$ placed on the edges of an equilateral triangle also form a relative equilibrium. Explicitly, choose $z_1,z_2,z_3\in\R^2\setminus\{0\}$ such that $|z_1-z_2|=|z_1-z_3|=|z_2-z_3|=:s$ and $\Ga_1z_1+\Ga_2z_2+\Ga_3z_3=0$. Then $Z_k(t):=e^{-\om Jt}z_k$, $k=1,2,3$, with $\om=\frac{\Ga_1+\Ga_2}{\pi s^2}$ solves ($HS_0$). It is non-degenerate provided the total vortex angular momentum $L=\Ga_1\Ga_2+\Ga_1\Ga_3+\Ga_2\Ga_3$ and the total vorticity $\Ga=\Ga_1+\Ga_2+\Ga_3$ satisfy
  \[
  \Ga\neq 0,\quad L\neq 0\quad \text{and}\quad L\neq \Ga_1^2+\Ga_2^2+\Ga_3^2.
  \]
  These statements are well known and can be found in \cite{Roberts:2013}.

  c) As in \cite{Bartsch-Gebhard:2016} one can refine Theorem \ref{thm:existence} by including symmetries provided the subgroup
  \[
  \Sym(\Ga):=\set{\si\in\Si_N:(\Ga_{\si(1)},\ldots,\Ga_{\si(N)})=(\Ga_1,\dots,\Ga_N)}%\ne\{\id\}.
  \]
  of the group $\Si_N$ of permutations of $1,\dots,N$ is non-trivial. Given $\si\in\Sym(\Ga)$ with ${\text{ord}}(\si)=k$ we look for solutions of ($HS$) with the following spatio-temporal symmetry:
  \[
  \big(u_1(t),\dots,u_N(t)\big) = \big(u_{\si^{-1}(1)}(t+\frac{2\pi}{k}),\ldots,u_{\si^{-1}(N)}(t+\frac{2\pi}{k})\big) =: \si*u(t).
  \]
  The space of such function is denoted by $X^\si:= \set{u\in H^1_{2\pi}:\si*u=u}$. It is a closed subspace of $H^1_{2\pi}$. If $Z\in X^\si$ is nondegenerate within this space, i.e.\ \eqref{eq:R2lin} has exactly three linearly independent $\frac{2\pi}{\abs{\om}}$-periodic solutions in $X^\si$, then one can work in $X^\si$ instead of $H^1_{2\pi}$. Since this extension is straightforward we leave it to the reader. Observe that it allows to treat the case where $Z$ is Thomson's solution of $N$ identical vortices (i.e.\ $\Sym(\Ga)=\Si_N$) placed on the vertices of a regular $N$-gon.

  d) Using a special degree it was proved in \cite{Bartsch-Gebhard:2016} that there exists a global continuum of periodic solutions of $(HS)$ extending the family $z^{(r)}$, $0<r<r_0$.
\end{rem}

\section{Proof of Theorem \ref{thm:existence}}\label{sec:proof}

We assume without loss of generality that $a_0=0$. First we introduce a parameter $r>0$ in order to normalize the period to $2\pi$ and in order to obtain ($HS_0$) as a limit of a corresponding family of systems ($HS_r$). More precisely, we consider the Hamiltonian
\[
H_r(u) := H_0(u)-F(ru)+F(0) = H(ru)+\frac{1}{2\pi}\sum_{j\neq k}\Ga_j\Ga_k\log r+F(0).
\]
Then $z$ solves $(HS)$ iff $u(s)=\frac1r z(r^2s)$ solves
\[
\quad \Ga_k\dot{u}_k = J\nabla_{u_k} H_r(u),\quad k=1,\dots,N. \tag{$HS_r$}
\]
For $r\ge0$ the Hamiltonian $H_r$ is defined on
\[
\cO_r=\{u\in\R^{2N}: r u_k\in\Om, u_j\neq u_k\text{ for all } j\neq k\}.
\]
Observe that $H_r(u)\to H_0(u)$ as $r\to 0$, moreover $\cO_r=\cF_N(\frac1r\Om)$ for $r>0$ and $\cO_0=\cF_N(\R^2)$. Thus this blow-up procedure produces a continuous family of Hamiltonians depending on $r\ge0$, and we may think of ($HS_0$) as a limit of ($HS_r$).

We shall seek $2\pi$-periodic solutions $u\in H^1_{2\pi}$ of $(HS_r)$, corresponding to $2\pi r^2$-periodic solutions $z$ of $(HS)$. Setting
\[
M_\Ga =
\begin{pmatrix}
\Ga_1&&\\
&\ddots&\\
&& \Ga_N
\end{pmatrix}\otimes E_2
=\begin{pmatrix}
\Ga_1E_2&&\\
&\ddots&\\
&& \Ga_NE_2
\end{pmatrix}
\in\R^{2N\times 2N}
\]
and
$J_N = E_N\otimes J\in \R^{2N\times 2N}$, where $E_m\in\R^{m\times m}$ is the identity matrix, the functional associated to $(HS_r)$ is
\[
\fJ_r(u) = \frac12 \int^{2\pi}_0 M_\Ga\dot{u}\cdot J_Nu - \int^{2\pi}_0 H_r(u)
 = \fJ_0(u)-\int^{2\pi}_0 F(ru)+2\pi F(0).
\]
This is defined on
$$
\La_r:=\{u\in H^1_{2\pi}:u(t)\in\cO_r\text{ for all }t\in\R\},
$$
which is an open subset of $H^1_{2\pi}$. The solution $Z\in H^1_{2\pi}$ of $(HS_0)$ generates a $3$-dimensional manifold $\cM=\cM_Z$ of solutions obtained from $Z$ by time shifts $\theta*Z$, $\theta\in S^1$, and translations $(a+Z_1, \ldots, a+Z_N)$, $a\in\R^2$. Writing $\hat{a}=(a,\ldots,a)\in\R^{2N}$ for $a\in\R^2$ we have
\[
\cM=\{\hat{a}+\theta*Z:a\in\R^2,\ \theta\in S^1\}\subset H^1_{2\pi}.
\]
Setting $D:=\{\hat{a}:a\in\R^2\}\subset\R^{2N}\subset H^1_{2\pi}$ the tangent space of $\cM$ at $Z$ is $T_Z=D\oplus\R\dot{Z}$. The assumption that $Z$ is a nondegenerate solution of $(HS_0)$ means that $\cM$ is a nondegenerate critical manifold of $\fJ_0$, i.e.\ $\fJ''_0(Z):H^1_{2\pi}\to H^1_{2\pi}$ has kernel $T_Z$ and image $N_Z=(T_Z)^\perp$, hence it induces an isomorphism $N_Z \to N_Z$.

We define $\Phi_r(v):=\nabla\fJ_r(Z+v)$ for $v\in H^1_{2\pi}$ close to $0$, and want to solve $\Phi_r(v)=0$ for $r\approx 0$, $v\approx 0$. There holds
\[
\Phi_r(v) = \nabla \fJ_0(Z+v) - r(\id - \De)^{-1}(\nabla F(rZ+rv))
 =:\nabla \fJ_0(Z+v)-\Psi_r(v)
\]
where $\id-\De: H^{s+2}_{2\pi} \to H^s_{2\pi}$, $w \mapsto w-\ddot{w}$, which is an isomorphism for any $s\ge0$. For a closed subspace $Y\subset H^1_{2\pi}$ let $P_Y:H^1_{2\pi}\to Y$ be the orthogonal projection. We consider $X:=\dot{Z}^\perp\subset H^1_{2\pi}$ and its direct sum decomposition $X=N_Z\oplus D$.

\begin{lem}\label{lem:reduction}
There exists $\veps >0$ and $r_\veps>0$ such that for $0<r<r_\veps$ the equation $P_X(\Phi_r(v))=0$ has a unique solution $v^{(r)}\in X$ with $\norm{v^{r}}\le\veps$. Moreover, $\norm{v^{(r)}}=O(r)$ as $r\to 0$, and $v^{(r)}$ is $C^1$ in $r$.
\end{lem}

\begin{proof}
We define the linear operator $L_r:X\to X$ by
\[
\bal
L_r[v] &= P_X\circ D\Phi_r(0)[v] = \fJ''_0(Z)[v]-D\Psi_r(0)[v]\\
{}&= \fJ''_0 (Z)[v]-r^2(id-\De)^{-1}\big(F''(rZ)[v]\big).
\eal
\]
Recall that $\fJ''_0(Z)$ defines an isomorphism $N_Z\to N_Z$. With respect to the decomposition $X=N_Z\oplus D$ we have the block decomposition
\[
L_r=\begin{pmatrix} A_r & -r^2B_r \\ -r^2C_r & -r^2D_r \end{pmatrix}
%\left(\bal A_r && -r^2B_r\\ -r^2C_r && -r^2B_r\eal\right)
\]
i.e.\ for $v\in X$ with $\hat{v}:=P_{N_Z}[v]$, $\hat{a}:=P_D[v]$ there holds
$L_r[v] = A_r[\hat{v}]-r^2 B_r[\hat{a}]-r^2C_r[\hat{v}]-r^2 D_r[\hat{a}]$. We also have
\[
\bal
&A_r[\hat{v}] = \fJ''_0(Z)[\hat{v}]-r^2P_{N_Z}D\Psi_r(0)[u]
 \stackrel{r\to 0}{\longrightarrow} \fJ''_0(Z)[\hat{v}],\\
&B_r[\hat{a}] = P_{N_Z}(\id-\De)^{-1}\big(F''(rZ)[\hat{a}]\big)
 \stackrel{r\to 0}{\longrightarrow} P_{N_Z} (\id-\De)^{-1}\big(F''(0)[\hat{a}]\big) =: B_0[\hat{a}],\\
&C_r[\hat{u}] = P_D(\id-\De)^{-1} \big(F''(rZ)[\hat{v}]\big)
 \stackrel{r\to 0}{\longrightarrow} P_D(\id-\De)^{-1}\big(F''(0)[\hat{u}]\big) = :C_0[\hat{v}],\\
&D_r[\hat{a}] = P_D(\id-\De)^{-1}\big(F''(rZ)[\hat{a}]\big)
 \stackrel{r\to 0}{\longrightarrow} P_D (\id-\De)^{-1}\big(F''(0)[\hat{a}]\big) =: D_0[\hat{a}].
\eal
\]

We claim that $D_0:D\to D$ is an isomorphism. First observe that
$D_0[\hat{a}] = P_D\big(F''(0)[\hat{a}]\big)$. Using the fact that $g(w,z)=g(z,w)$ is symmetric, a straightforward computation yields
\beql{eq:Fsecond}
F''(0)=\frac12(\Ga_j \Ga_k)_{j,k=1,\ldots, N}\otimes h''(0)
\eeq
i.e.\ for $c=(c_1, \ldots, c_N)\in\R^{2N}$ the $j$-th compact of $F''(0)[c]$ is given by
\[
\big(F''(0)[c]\big)_j = \frac12\sum_{k=1}^N \Ga_j\Ga_k h''(0)[c_k].
\]
In particular, for $a\in\R^2$ and $c=\hat{a}$ we obtain
\[
\big(F''(0)[\hat{a}]\big)_j = \frac12\Ga_j\Ga h''(0)[a].
\]
In order to compute $P_D\big(F''(0)[\hat{a}]\big)$ let $e_1, e_2\in\R^2$ be the standard basis. Then:
\[
\langle F''(0)[\hat{a}],\hat{e}_i\rangle_{H^1_{2\pi}}
 = 2\pi\cdot\frac12\sum_{j=1}^N\Ga_j\Ga \langle h''(0)[a], e_i\rangle_{\R^2}
 = \pi\Ga^2 \langle h''(0)[a], e_i\rangle_{\R^2}\\
\]
and therefore, using $\norm{e_i}^2_{H^1}=2\pi N$,
$$
P_D(F''(0)[\hat{a}]) = \frac{\Ga^2}{2N} \widehat{h''(0)[{a}]}.
$$
Since $h''(0)\in\R^{2\times 2}$ is invertible and $\Ga\ne0$, $D_0$ is invertible.

It follows that $D_r$ is invertible for $r\geq 0$ small. Moreover, $A_r$ is invertible for $r\geq 0$ small because $A_r\to A_0=J''(Z)|_{N_Z}$. Consequently $L_r$ is invertible for $r>0$ small and
\beql{eq:L_r-inverse}
L_r^{-1}
 = \begin{pmatrix} A_r^{-1}+r^2O(1) & -A^{-1}_0 B_0 D^{-1}_0+o(1) \\
         -D_0^{-1} C_0 A_0+o(1) & \frac{1}{r^2} (D_0^{-1}+o(1))
   \end{pmatrix}
%\left(\bal &A^{-1}_v+r^20(1)\quad &&-A^{-1}_0 B_0 D^{-1}_0+0(1)\\ -&D^{-1}_0 C_0 A_0+0(1)\quad && \frac{1}{r^2} (D^{-1}+0(1)) \eal\right)
\eeq
with respect to the decomposition $X=N_Z\oplus D$.

Writing
\beql{eq:P_XPhi}
P_X\Phi_r(v)=P_X\Phi_r(0)+L_r[v]-\varphi_r(v),
\eeq
the equation $P_X\Phi_r(v)=0$ is equivalent to
\beql{eq:fixedpoint}
v = -L_r^{-1}[P_X\Phi_r(0)]+L_r^{-1}[\varphi_r(v)] =: K_r(v).
\eeq
We claim that
\beql{eq:P_XPhi=O(r)}
L^{-1}_r[P_X\Phi_r(0)]=O(r)\text{ as }r\to 0.
\eeq
From
\[
\Phi_r(0) = \Phi _0(0)-\psi_r(0) = -r(\id-\De)^{-1}\big(\nabla F(rZ)\big) = O(r)
\]
and \eqref{eq:L_r-inverse} we deduce
\beql{eq:P_Nz=O(r)}
P_{N_Z}\big(L_r^{-1}\Phi_r(0)\big) = O(r).
\eeq
Moreover,
\[
\frac{1}{r^2}\Phi_r(0)
 = -(\id-\De)^{-1}\left(\frac1r\nabla F(rZ)\right) = -(\id-\De)^{-1}\big(F''(0)[Z]\big)+O(r)
\]
as $r\to 0$ because $\nabla F(0)=0$ as a consequence of $\nabla h(0)=0$. For $c\in\R^N$ we have
\[
\langle (\id-\De)^{-1}(F''(0)[Z],c\rangle_{H^1} = \langle F''(0)[Z],c\rangle_{L^2}
 = \sum_{j=1}^N\langle(F''(0)[Z])_j,c_j\rangle_{L^2}
\]
and $(F''((0)[Z])_j=\frac12\Ga_j\sum_{k=1}^N\Ga_kh''(0)[Z_k]$
by \eqref{eq:Fsecond}. Finally, using that the center of vorticity of $Z$ is $0$ we obtain
\[
\int_0^{2\pi}\sum_{k=1}^N\Ga_k\langle h''(0)[Z_k],c_j\rangle_{\R^2}
 = \int_0^{2\pi}\left\langle\sum_{k=1}^N\Ga_k Z_k,h''(0)c_j\right\rangle_{\R^2}=0.
\]
Consequently, $P_D\big[(\id-\De)^{-1}(F''(0)[Z])\big]=0$ which implies that
\beql{eq:P_DPhi=O(r^3)}
\frac{1}{r^z}P_D\Phi_r(0)=O(r)\quad\text{for } r\to 0.
\eeq
Now \eqref{eq:P_XPhi=O(r)} follows from \eqref{eq:L_r-inverse}, \eqref{eq:P_Nz=O(r)}, and \eqref{eq:P_DPhi=O(r^3)}.

Next we claim that
\beql{eq:L_r-inverse=O(rv)}
L_r^{-1}D\varphi_r(v) = O(r) O(\norm{v})\quad\text{for }r\to 0, v\to 0.
\eeq
In order to see this observe that
\[
D\varphi_r(v) = -\De\Psi_r(v)-D\Psi_r(0) = r^2(\id-\De)^{-1}\big(F''(rZ+rv)-F''(rZ)\big).
\]
Since $F''(rZ+rv)-F''(rZ)=O(r)O(\norm{v})$ as $r\to 0$, $v\to 0$, we obtain $\frac{1}{r^2}D\varphi_r(v)=O(r)O(\norm{v})$ as $r\to 0$, $v\to 0$. Using \eqref{eq:L_r-inverse} this yields \eqref{eq:L_r-inverse=O(rv)}.

Now we can solve \eqref{eq:fixedpoint} by the contraction mapping principle. Due to \eqref{eq:L_r-inverse=O(rv)} there exists $\rho>0$ and $\veps>0$ such that
$\norm{L_r^{-1} D\varphi_r(v)}\le\frac12$ for $0<r\le\rho$, $\norm{v}\le\veps$. By \eqref{eq:P_XPhi=O(r)} there exists $r_\veps>0$ such that $\norm{L_r^{-1}P_X\Phi_r(0)}\le\frac{\veps}{2}$ for $0<r<r_\veps$. Then the map $K_r$ from \eqref{eq:fixedpoint} satisfies:
\[
\norm{K_r(v)-K_r(w)} = \norm{L_r^{-1}(\varphi_r(v)-\varphi_r(w))} \le \frac12\norm{v-w}
\]
for all $v,w\in X$ with $\norm{v},\norm{w}\le\veps$. In addition,
\[
\norm{K_r(v)}
 \leq \norm{K_r(v)-K_r(0)}+\norm{K_r(0)} \le \frac12\norm{v}+\frac{\veps}{2} \le \veps
\]
if $\norm{v}\le\veps$. Therefore $K_r$ induces a contraction on
$B_\veps=\{v\in X:\norm{v}\leq\veps\}$, hence there exists a unique fixed point
$v^{(r)}\in B_\veps$ for $K_r$, $0<r<r_\veps$. That $v^{(r)}$ is $C^1$ in $r$ follows from the implicit function theorem because $\Phi_r(v)$ is $C^1$ in $r$ and $v$.

It remains to prove that $\norm{v^{(r)}}=O(r)$ as $r\to0$. This follows from \eqref{eq:P_XPhi=O(r)} and \eqref{eq:L_r-inverse=O(rv)}:
\[
\norm{v^{(r)}} = \norm{K_r(v^{(r)})}
 \le \norm{L_r^{-1}[P_X\Phi_r(0)]}+\norm{L_r^{-1}[\varphi_r(v)^{(r)}]} = O(r)
\]
\end{proof}

\begin{lem}\label{lem:equivariance}
For $r\in(0,r_\veps)$ the map $S^1\to H^1_{2\pi}$, $\theta \mapsto\theta*v^{(r)}$, is of class $C^1$, and
$\langle\dot{v}^{(r)},\dot{Z}\rangle = \frac{d}{d\theta}\langle\theta*v^{(r)},\dot{Z}\rangle
 = O(r)$.
\end{lem}

Observe that for an arbitrary element $u\in H^1_{2\pi}$ the map $S^1\to H^1_{2\pi}$, $\theta\mapsto\theta*u$, may not be differentiable.

\begin{proof}
Since $\nabla\fJ_r$ is $S^1$-equivariant, the unique solution
$v\in X_\theta :=(\theta*\dot{Z})^\perp$ of
\beql{eq:theta}
P_{X_\theta}[\nabla\fJ_r(\theta*Z+v)]=0, \ \ \langle v,\theta*\dot{Z} \rangle=0, \ \ \norm{v}\le\veps,
\eeq
is $v=\theta*v^{(r)}$. This uses Lemma \ref{lem:reduction} with $Z$ replaced by $\theta*Z$. The $S^1$-invariance of $\fJ_r$ implies that $\veps$ and $r_\veps$ are independent of
$\theta\in S^1$. Consider the map $f_r:S^1\times H^1_{2\pi}\to H^1_{2\pi}$, defined by
\[
f_r(\theta,v) := P_{X_\theta}[\nabla\fJ_r(\theta*Z+v)]
  + \frac{\langle\theta*\dot{Z},v\rangle}{\|\dot{Z}\|^2}\theta*\dot{Z}.
\]
This is of class $C^1$ and
\[
D_v f_r(0,v) = P_X\circ D\Phi_r(v)+P_{\R\dot{Z}}: H^1_{2\pi} \to H^1_{2\pi}.
\]
Now \eqref{eq:P_XPhi} and \eqref{eq:L_r-inverse=O(rv)} imply that
$P_X\circ D\Phi_r(v)|_X:X\to X$ is invertible for $0<r<r_\veps$, $\norm{v}\leq \veps$, hence $D_v f_r(0,v)$ is invertible for $0<r<r_\veps$, $\norm{v}\le\veps$. It follows that the solution $\theta*v^{(r)}$ of $f_r(\theta,v)=0$ is $C^1$ in $\theta$.

Finally, $\langle\dot{v}^{(r)},\dot{Z}\rangle = O(r)$ as $r\to 0$ follows by differentiating the identity $\langle\theta*v^{(r)},\theta*\dot{Z}\rangle =0$ and using Lemma~\ref{lem:reduction}:
\[
0 = {\frac{d}{d\theta}}\bigg|_{\theta=0}\langle\theta*v^{(r)},\theta*\dot{Z}\rangle
  = \langle\dot{v}^{(r)},\dot{Z}\rangle+\langle v^{(r)},\ddot{Z}\rangle
  = \langle\dot{v}^{(r)},\dot{Z}\rangle+O(r)
\]
\end{proof}

Now we easily conclude the

\begin{altproof}{Theorem~\ref{thm:existence}}
We differentiate the identity
\[
\fJ_r(\theta*Z+\theta*v^{(r)}) = \fJ_r(Z+v^{(r)})
\]
and obtain:
\[
\bal
0&= \frac{d}{d\theta}\bigg|_{\theta=0} \fJ_r(\theta*Z+\theta*v^{(r)})
  = \langle\nabla \fJ_r(Z+v^{(r)}),\dot{Z}+\dot{v}^{(r)}\rangle\\
{}&= \langle\nabla \fJ_r(Z+v^{(r)}),\dot{Z}\rangle
     + \frac{\langle\dot{v}^{(r)},\dot{Z}\rangle}{\norm{\dot{Z}}^2}
        \cdot \langle\nabla \fJ_r(Z+v^{(r)}),\dot{Z} \rangle.
\eal
\]
The last equality uses $P_X\nabla \fJ_r(Z+v^{(r)})=0$. Now $\langle\dot{v}^{(r)},\dot{Z}\rangle=O(r)$ by Lemma~\ref{lem:equivariance}, hence
\[
0 = \langle\nabla \fJ_r(Z+v^{(r)}),(1+O(r))\dot{Z}\rangle.
\]
This implies $\langle\nabla\fJ_r(Z+v^{(r)}),\dot{Z}\rangle=0$, hence $\nabla\fJ_r(Z+v^{(r)})=0$, for $0<r<r_0$ small. This proves \ref{thm:existence}~a) with $u^{(r)}:=Z+v^{(r)}$.

The $S^1$-invariant neigborhood $\cU\subset H^1_{2\pi}$ of $S^1*Z$ in \ref{thm:existence}~b) is simply the $\veps$-neighborhood of $S^1*Z$ with $\veps$ from Lemma~\ref{lem:reduction}.
\end{altproof}

% ------------------------------------------------------------------------
\end{document}